\documentclass[11pt]{article}
\usepackage{amsmath, amssymb, amsthm}

\date{}
\title{\vspace{-0.8cm}Rainbow Tur\'{a}n Problem for Even Cycles}
\author{
Shagnik Das \thanks{Department of Mathematics, UCLA, Los Angeles, CA, 90095. Email: shagnik@ucla.edu.}
\and
Choongbum Lee \thanks{Department of Mathematics, UCLA, Los
Angeles, CA, 90095. Email: choongbum.lee@gmail.com. Research
supported in part by a Samsung Scholarship.}
\and
Benny Sudakov \thanks{Department of Mathematics, UCLA, Los Angeles, CA 90095.
Email: bsudakov@math.ucla.edu. Research supported in part by NSF
CAREER award DMS-0812005 and by a USA-Israel BSF grant. }
}

\theoremstyle{plain}
\newtheorem{THM}{Theorem}[section]
\newtheorem{PROP}[THM]{Proposition}
\newtheorem{LEMMA}[THM]{Lemma}

\newtheorem{COR}[THM]{Corollary}

\theoremstyle{definition}

\newcommand{\BFP}{\mathbf{P}}
\newcommand{\BBE}{\mathbb{E}}

\begin{document}
\maketitle

\begin{abstract}
An edge-colored graph is rainbow if all its edges are
colored with distinct colors.
For a fixed graph $H$, the rainbow Tur\'{a}n number
$\mathrm{ex}^{\ast}(n,H)$ is defined as the maximum number of edges
in a properly edge-colored graph on $n$ vertices with no rainbow copy of $H$.  We study the
rainbow Tur\'{a}n number of even cycles, and prove that for every fixed $\varepsilon > 0$, there is a constant $C(\varepsilon)$ such that every
properly edge-colored graph on $n$ vertices with at least
$C(\varepsilon) n^{1 + \varepsilon}$ edges contains a rainbow cycle of
even length at most $2 \left \lceil \frac{\ln 4 - \ln \varepsilon}{\ln ( 1 + \varepsilon )} \right \rceil$.
This partially answers a question of Keevash, Mubayi, Sudakov, and
Verstra\"ete, who asked how dense a graph can be without having a rainbow cycle of any length.
\end{abstract}

\section{Introduction}

An edge-colored graph is \emph{rainbow} if all its edges have distinct colors.  The rainbow Tur\'{a}n problem, first introduced by Keevash, Mubayi, Sudakov and
Verstra\"{e}te \cite{KMSV}, asks the following question: given a fixed graph $H$, what is the maximum
number of edges in a properly edge-colored graph $G$ on $n$ vertices
with no rainbow copy of $H$?  This maximum is denoted
$\mathrm{ex}^{\ast}(n,H)$, and is called the rainbow Tur\'{a}n
number of $H$. In this paper, we study the rainbow Tur\'{a}n problem for even
cycles.

\subsection{Background}

The rainbow Tur\'{a}n problem has a certain aesthetic appeal, as it
lies at the intersection of two key areas of extremal graph theory.
On the one hand we have the classical Tur\'{a}n problem, which, for
a given graph $H$, asks for the maximum number of edges in an
$H$-free graph on $n$ vertices.  This maximum, the Tur\'{a}n number
of $H$, is denoted by $\mathrm{ex}(n,H)$, and determining it is one
of the oldest problems in extremal combinatorics.  Tur\'{a}n
\cite{Tur} solved the problem for cliques by finding
$\mathrm{ex}(n,K_k)$. Erd\H{o}s and Stone \cite{ES} then found the
asymptotics of $\mathrm{ex}(n,H)$ for all non-bipartite graphs $H$.
The problem of determining the Tur\'{a}n numbers of bipartite graphs
is still largely open.  Of particular interest is the case of even
cycles. Erd\H{o}s conjectured that $\mathrm{ex}(n,C_{2k}) = \Theta(
n^{1 + \frac{1}{k}} )$.  Bondy and Simonovits \cite{BS} gave the
corresponding upper bound, but as of yet a matching lower bound is
only known for $k = 2, 3,$ or $5$.

On the other hand, there is a great deal of literature on extremal
problems regarding (not necessarily proper) edge-colored graphs. The
Canonical Ramsey Theorem of Erd\H{o}s and Rado \cite{ER} shows, as a
special case, that when $n$ is large with respect to $t$, then any
proper edge-coloring of $K_n$ contains a rainbow $K_t$. Another
variation is when one allows at most $k$ colors to be used for edges
incident to each vertex. This notion, called local $k$-colorings,
has been first introduced by Gy\'arf\'as, Lehel, Schelp, and Tuza
\cite{GLST}, and has been studied in a series of works. More
recently, Alon, Jiang, Miller and Pritikin \cite{AJMP} studied the
problem of finding a rainbow copy of a graph $H$ in an edge-coloring
of $K_n$ where each color appears at most $m$ times at any vertex.
The rainbow Tur\'{a}n problem is a Tur\'{a}n-type extension in the
case $m = 1$.  From this point on, we shall only consider proper
edge-colorings.

The rainbow Tur\'{a}n problem for even cycles is of particular
interest because of the following connection to a problem in number
theory, as noted in \cite{KMSV}.  Given an abelian group $\Gamma$, a
subset $A$ is called a $B_k^{\ast}$-set if it does not contain
disjoint $k$-sets $B,C$ with the same sum.  Given a set $A$, we form
a bipartite graph $G$ as follows: the two parts $X$ and $Y$ are
copies of $\Gamma$, and we have an edge from $x \in X$ to $y \in Y$
if and only if $x - y \in A$. Moreover, the edge $xy$ is given the
color $x - y \in A$.  It is easy to see that this is a proper
edge-coloring of a graph with $|\Gamma| |A|$ edges, and $A$ is a
$B_k^{\ast}$-set precisely when $G$ has no rainbow $C_{2k}$.  Hence
bounds on $B_k^{\ast}$-sets give bounds on
$\mathrm{ex}^{\ast}(n,C_{2k})$, and vice versa.

\subsection{Known Results}

Note that we trivially have the lower bound $\mathrm{ex}(n,H) \le
\mathrm{ex}^{\ast}(n,H)$, since if a graph is $H$-free, then it is rainbow-$H$-free under any proper edge coloring.
One is thus generally interested in either
finding a matching upper bound, or showing that $\mathrm{ex}^{\ast}(n,H)$ is asymptotically larger than $\mathrm{ex}(n,H)$.  In the original paper of Keevash,
Sudakov, Mubayi and Verstra\"{e}te \cite{KMSV}, this problem was
resolved for a wide range of graphs.  In particular, it was shown
that for non-bipartite $H$, the Rainbow Tur\'{a}n problem can be
reduced to the Tur\'{a}n problem, and as a result
$\mathrm{ex}^{\ast}(n,H)$ is asymptotically (and in some cases
exactly) equal to $\mathrm{ex}(n,H)$.  For bipartite $H$ with a
maximum degree of $s$ in one of the parts, they found an upper bound
of $\mathrm{ex}^{\ast}(n,H) = O(n^{2 - \frac{1}{s}})$.  This matches
the general upper bound for Tur\'{a}n numbers of such graphs, and in
particular is tight for $C_4$ (where $s = 2$).

An interesting case which is not implied by the above mentioned
results is the case of even cycles of length at least $6$, and
special attention was paid to this case, in light of the connection
to $B_k^{\ast}$-sets discussed earlier.  Using Bose and Chawla's
\cite{BC} construction of large $B_k^{\ast}$-sets, the authors gave
a lower bound of $\mathrm{ex}^{\ast}(n,C_{2k}) = \Omega (n^{1 +
\frac{1}{k}})$ - this is better than the best known bound for
$\mathrm{ex}(n,C_{2k})$ for general $k$.  A matching upper bound was
obtained in the case of the six-cycle $C_6$, so it is known that
$\mathrm{ex}^{\ast}(n,C_6) = \Theta (n^{1 + \frac{1}{3}})$. However,
surprisingly, $\mathrm{ex}^{\ast}(n,C_6)$ is asymptotically larger
than $\mathrm{ex}(n,C_6)$.

Another problem considered was that of rainbow acyclicity - what is
the maximum number of edges in an edge-colored graph on $n$ vertices
with no rainbow cycle of any length?  Let $f(n)$ denote this
maximum.  In the uncolored setting, the answer is given by a tree,
which has $n-1$ edges.  However, as described in \cite{KMSV},
coloring the $d$-dimensional hypercube with $d$ colors, where
parallel edges get the same color, gives a rainbow acyclic proper
edge-coloring, and hence $f(n) = \Omega( n \ln n )$. The best known
upper bound to date was $f(n) = O(n^{1+\frac{1}{3}})$, which follows
from the bound $\mathrm{ex}^{\ast}(n,C_6) = \Theta (n^{1 +
\frac{1}{3}})$.

Keevash, Mubayi, Sudakov, and Verstra\"ete listed these two
questions, determining $\mathrm{ex}(n,C_{2k})$ and $f(n)$, as
interesting open problems in the study of rainbow Tur\'an numbers.

\subsection{Our Results}

In this paper we improve the upper bound on the rainbow Tur\'{a}n
number of even cycles, and make progress towards the two open problems
mentioned in the previous subsection. Following is the main
theorem of this paper:

\begin{THM} \label{thm:main}
For every fixed $\varepsilon > 0$ there is a constant
$C(\varepsilon)$ such that any properly edge-colored graph on $n$
vertices with at least $C(\varepsilon) n^{1 + \varepsilon}$ edges
contains a rainbow copy of an even cycle of length at most $2k$,
where $k = \left \lceil \frac{\ln 4 - \ln \varepsilon} { \ln (1 +
\varepsilon) } \right \rceil $.
\end{THM}

Our result easily gives an upper bound on the size of
rainbow acyclic graphs.\footnotemark

\footnotetext{As we remark in the concluding section, one can do somewhat better than this corollary.}

\begin{COR} \label{cor:acyclic}
Let $f(n)$ denote the size of the largest properly edge-colored
graph on $n$ vertices that contains no rainbow cycle.  Then for any
fixed $\varepsilon > 0$ and sufficiently large $n$, we have $f(n) <
n^{1 + \varepsilon}$.
\end{COR}

With a little more work, we can show that a graph satisfying the
condition of Theorem \ref{thm:main} must contain a rainbow cycle of
length exactly $2k$.
Therefore, inverting the relationship between $k$ and $\varepsilon$
gives a bound on $\mathrm{ex}^{\ast}(n,C_{2k})$.

\begin{COR} \label{cor:rtc2k}
For every fixed integer $k \ge 2$, $\mathrm{ex}^{\ast}(n,C_{2k}) =
O\left( n^{1 + \frac{(1 + \varepsilon_k) \ln k}{k}} \right)$, where
$\varepsilon_k \rightarrow 0$ as $k \rightarrow \infty$.
\end{COR}

\subsection{Outline and Notation}

This paper is organized as follows.  Section 2 provides a couple of
quick probabilistic lemmas.  The proof of Theorem \ref{thm:main} is
then given in Section 3, although the proof of the key proposition
is deferred until Section 4.  The final section contains some
concluding remarks and open problems.

A graph $G$ is given by a pair of vertex set $V(G)$ and edge set $E(G)$.
For a vertex $v \in V(G)$, we use $d(v)$ to denote its degree, and
for a subset of vertices $X$, we let $d(v,X)$ be the number of neighbors
of $v$ in the set $X$. We use the notation $ \mathrm{Bin}(n,p)$ to denote a
binomial random variable with parameters $n$ and $p$. Throughout the
paper $\log$ is used for the logarithm function of base 2, and $\ln$
is used for the natural logarithm.

\section{Preliminary Lemmas}

In this section we will prove a couple of technical lemmas that will be used in our proof of Theorem \ref{thm:main}.  Both will be proven using the probabilistic method, and will rely on the following form of Hoeffding's Inequality as appears in \cite[Theorem 2.3]{Mcd}.

\begin{THM} \label{thm:conc}
Let the random variables $X_1, X_2, \hdots, X_k$ be independent, with $0 \le X_i \le 1$ for each $i$.  Let $S = \sum_{i=1}^k X_i$, and $\mu = \BBE [ S ]$.  Then for any $s \le \frac{1}{2} \mu$ and $t \ge 2 \mu$, we have
\[ \BFP ( S \le s ) \le \mathrm{exp} \left( - \frac{s}{4} \right) \qquad \textrm{and} \qquad \BFP ( S \ge t ) \le \mathrm{exp} \left( - \frac{3t}{16} \right) . \]
\end{THM}

Our first lemma asserts that for any edge-colored graph with large minimum degree, the colors of the graph can be partitioned into disjoint classes in such a way that for every color class, the edges using colors from that class form a subgraph with large minimum degree.

\begin{LEMMA} \label{lem:split}
Let $G$ be an edge-colored graph on $n$ vertices with minimum degree $\delta$, and let $k$ be a positive integer.  Let $\mathcal{C}$ be the set of colors in $G$.  If $nk \; \mathrm{exp} \left( - \frac{\delta}{8k} \right) < 1$, then there is a partition $\mathcal{C} = \bigsqcup_{i=1}^k \mathcal{C}_i$ such that for every vertex $v$ and color class $\mathcal{C}_i$, $v$ has at least $\frac{\delta}{2k}$ edges with colors from $\mathcal{C}_i$.
\end{LEMMA}

\begin{proof}
Independently and uniformly at random assign each color $c \in \mathcal{C}$ to one of the $k$ color classes $\mathcal{C}_i$.  We will show that the resulting partition has the desired property with positive probability.

Fix a vertex $v$ and a color class $\mathcal{C}_i$.  Let $d(v)$ be the degree of $v$ in $G$, and let $d_{v,i}$ denote the number of edges incident to $v$ that have a color from $\mathcal{C}_i$.  Note that the color of every edge is in $\mathcal{C}_i$ with probability $\frac{1}{k}$.  Moreover, since the coloring is proper, the edges incident to $v$ have distinct colors, and hence are in $\mathcal{C}_i$ independently of one another.  Thus $d_{v,i} \sim \mathrm{Bin}\left( d(v), \frac{1}{k} \right)$, and $\BBE [ d_{v,i} ] = \frac{d(v)}{k} \ge \frac{\delta}{k}$ by our assumption on the minimum degree.

By Theorem \ref{thm:conc}, we have
\[ \BFP \left( d_{v,i} \le \frac{\delta}{2k} \right) \le \mathrm{exp} \left( - \frac{\delta}{8k} \right). \]
By a union bound,
\[ \BFP \left( \exists v, i : d_{v,i} \le \frac{\delta}{2k} \right) \le n k \; \mathrm{exp} \left( - \frac{\delta}{8k} \right) < 1, \]
and hence $\BFP \left( d_{v,i} > \frac{\delta}{2k} \; \forall v, i \right) > 0$.  Thus the desired partition exists.
\end{proof}

Given a set $X$ with a family of small subsets, the second lemma allows us to choose a subset of $X$ of specified size while retaining control over the sizes of the subsets.

\begin{LEMMA} \label{lem:shrink}
Let $\beta, \gamma \in (0,1)$ be parameters.  Suppose we have a set $X$ and a collection of subsets $X_j$, $1 \le j \le m$, such that $|X_j| \le \beta |X|$ for each $j$.  Provided $3m \; \mathrm{exp} \left( - \frac{1}{8} \beta \gamma |X| \right) < 1$, there exists a subset $Y \subset X$ with $\frac{1}{2} \gamma |X| \le |Y| \le 2 \gamma |X|$ such that for every $j$, we have $|X_j \cap Y| \le 4 \beta |Y|$.
\end{LEMMA}

\begin{proof}
Let $Y$ be the random subset of $X$ obtained by selecting each element independently with probability $\gamma$.  Let $Y_j = X_j \cap Y$.  Then we have $|Y| \sim \mathrm{Bin}(|X|, \gamma)$, and $|Y_j| \sim \mathrm{Bin}(|X_j|, \gamma)$.

By Theorem \ref{thm:conc},
\[ \BFP \left( |Y| \le \frac{1}{2} \gamma |X| \right) \le \mathrm{exp} \left( - \frac{1}{8} \gamma |X| \right), \qquad \textrm{and} \qquad \BFP \left( |Y| \ge 2 \gamma |X| \right) \le \mathrm{exp} \left( - \frac{3}{8} \gamma |X| \right). \]
Since $\BBE [ |Y_j| ] = \gamma |X_j| \le \beta \gamma |X|$, Theorem \ref{thm:conc} also gives
\[ \BFP \left( |Y_j| \ge 2 \beta \gamma |X| \right) \le \mathrm{exp} \left( - \frac{3}{8} \beta \gamma |X| \right). \]
By a union bound, the probability of any of these events occuring can be bounded by
\[ \mathrm{exp} \left( - \frac{1}{8} \gamma |X| \right) + \mathrm{exp} \left( - \frac{3}{8} \gamma |X| \right) + m \; \mathrm{exp} \left( - \frac{3}{8} \beta \gamma |X| \right) \le 3m \; \mathrm{exp} \left( - \frac{1}{8} \beta \gamma |X| \right) < 1. \]
Hence, with positive probability, none of these events occur.  In this case we have a subset $Y \subset X$ with $\frac{1}{2} \gamma |X| < |Y| < 2 \gamma |X|$ and $|X_j \cap Y| < 2 \beta \gamma |X| < 4 \beta |Y|$, as required.
\end{proof}

\section{Proof of the Main Theorem}

We will restrict our attention to bipartite graphs, and prove
Theorem \ref{thm:main} for bipartite graphs by using induction
within this class. The theorem for general graphs will then easily
follow since every graph contains a bipartite subgraph that contains
at least half of its original edges.

Our general strategy for proving Theorem \ref{thm:main} is as
follows.  We will choose an arbitrary vertex $v_0$, and grow a
subtree $T$ of $G$ rooted at $v_0$.  This subtree will have the
property that every path from $v_0$ in $T$ will be rainbow. The key
proposition will show that if $G$ has no short rainbow cycles, then
the levels of the tree must grow very rapidly, and will eventually
need to be larger than $G$, which is impossible.

In this section we formalize this argument, although the proof of
the key proposition is deferred to the next section.

\begin{proof}[Proof of Theorem \ref{thm:main}]
Fix $\varepsilon > 0$.  Without loss of generality, we may assume
$\varepsilon < \frac{1}{2}$, as otherwise the result follows from
the bound of $\mathrm{ex}^{\ast}(n,C_{2k}) = O \left(n^{2 -
\frac{1}{s}}\right)$ (with $s = 2$) given in \cite{KMSV}. We wish to
show there is a constant $C$ such that any edge-colored bipartite
graph $G$ on $n$ vertices with at least $C n^{1 + \varepsilon}$
edges contains a rainbow cycle of length at most $2k$, where $k =
\left \lceil \frac{\ln 4 - \ln \varepsilon}{ \ln (1 + \varepsilon)}
\right \rceil$.

We will prove this by induction on $n$.  For the base case, note
that if $n \le C$, then $C n^{1 + \varepsilon} > n^2$.  Hence there is
no graph on $n$ vertices with $C n^{1 + \varepsilon}$ edges, and so the
statement is vacuously true.  Thus by making the constant $C$ large,
we force $n$ to be large in the induction step below.  In
particular, we will require $C > 8k$ to be large enough that every $n \ge
C$ satisfies the following inequalities:
\[ n k \; \mathrm{exp} \left( - n^{\varepsilon} \right) < 1, \quad n^{\frac{1}{4} \varepsilon^3} > [4 (k+1)]^{2 + \varepsilon} \log n, \quad \textrm{and} \quad n^{\frac{1}{2} \varepsilon^2} > 2^{4 + (3k+2) \varepsilon} k^{2 + \varepsilon} ( \log n )^{1 + k \varepsilon}. \]

Now suppose $n > C$, and $G$ has at least $C n^{1 + \varepsilon}$
edges.  If $G$ has a vertex of degree at most $C n^{\varepsilon}$, then
by removing it we have a subgraph on $n-1$ vertices with at least $C
n^{1 + \varepsilon} - Cn^{\varepsilon} > C (n-1)^{1 + \varepsilon}$ edges.  By
induction, this subgraph contains a rainbow cycle of length at most
$2k$.  Hence we may assume $G$ has minimum degree at least $C n^{\varepsilon}$.

We now apply Lemma \ref{lem:split}.  By our bound on $C$, we have $n k \; \mathrm{exp} \left( - \frac{C n^{\varepsilon}}{8 k} \right) < 1$.  Hence we can split the colors into disjoint classes $\mathcal{C}_i$, $1 \le i \le k$, such that for each class $\mathcal{C}_i$, every vertex is incident to at least $\frac{C}{2k} n^{\varepsilon}$ edges of a color in $\mathcal{C}_i$.

\medskip

Let $v_0$ be an arbitrary vertex in $G$.  We will construct a subtree $T$ rooted at $v_0$, with vertices arranged in levels $L_i$, starting with $L_0 = \{ v_0 \}$.  Given a level $L_i$, the next level $L_{i+1}$ will be a carefully chosen subset of neighbors of $L_i$ using just the edges with colors from $\mathcal{C}_{i+1}$.  Note that this ensures that every vertex has a rainbow path back to $v_0$ in $T$.  Moreover, since every vertex in $L_i$ has a path of length $i$ back to $v_0$, and $G$ is bipartite, it follows that $L_i$ is an independent set in $G$.  It is useful to parametrize the size of the levels by defining $\alpha_i$ such that $|L_i| =n^{\alpha_i}$.

As mentioned above, every vertex $v \in T$ has a rainbow path back to $v_0$.  It will be important to keep track of which colors are used on this path.  Hence for every color $c$ and level $i$, we define $X_{i,c}$ to be the vertices in $L_i$ with an edge of color $c$ in their path back to $v_0$.  Since the path from $v$ to $v_0$ has length $i$, it follows that $\{ X_{i,c} \}_c$ forms an $i$-fold cover of $L_i$.  If we have a vertex $w \in L_{i+1}$ adjacent to $v_1, v_2 \in L_i$ with $v_1$ and $v_2$ using disjoint sets of colors on their paths back to $v_0$, this gives a rainbow cycle of length $2(i+1)$.
It turns out that forbidding such configurations gives large expansion from $L_i$ to $L_{i+1}$.

\medskip

The key proposition below formalizes the above observation and shows
that the levels grow quickly. As shown below, we will need to
maintain control over the sets $X_{i,c}$. To see the necessity of
this, suppose that we had $X_{i,c} = L_i$ for some $i$ and $c$. Then
every path through $L_i$ to $v_0$ would use the color $c$, and we
could not hope to find a rainbow cycle using our strategy. Note that
in the special case where the given graph is
$Cn^\varepsilon$-regular and the graph is colored using exactly
$Cn^\varepsilon$ colors, for every index $i$, there exists a color
$c$ such that $|X_{i,c}| \ge
\frac{|L_i|}{Cn^\varepsilon}=\Omega(n^{\alpha_i-\varepsilon})$. This
implies that we cannot hope for a upper bound on $|X_{i,c}|$ that is
better than $|X_{i,c}| = O(n^{\alpha_i-\varepsilon})$. The bound we
achieve in the following proposition is a poly-logarithmic factor
off this `optimal' bound.

\begin{PROP} \label{prop:key}
Given $1 \le i < k$, suppose that we are given sets $L_0, \cdots, L_i$ and sets $\{X_{i,c}\}_c$ satisfying the
following:
\begin{enumerate}
\item[(i)]
$|L_i| \ge \frac{1}{4}|L_j|$ for $0 \le j < i$, and
$\alpha_i \le 1 - \frac{1}{4} \varepsilon^2$, and
\item[(ii)] $|X_{i,c}| \le (8 \log n)^{i} n^{\alpha_i - \varepsilon}$ for all $c \in \mathcal{C}$.
\end{enumerate}
Then there is a set $L_{i+1}$ of neighbors of $L_i$ using colors from $\mathcal{C}_{i+1}$ such that:
\begin{enumerate}
    \item $\left( 1 + \frac{\varepsilon}{2} \right) - \alpha_{i+1} \leq \left( 1 + \varepsilon \right)^{-1} \left[ \left( 1 + \frac{\varepsilon}{2} \right) - \alpha_i \right]$, and
    \item for all colors $c$, we have $|X_{i+1,c}| \le (8 \log n)^{i+1} n^{\alpha_{i+1} - \varepsilon}$.
\end{enumerate}
Moreover, even if we have $(ii') \, |X_{i,c}| \le 4 (8 \log n)^{i} n^{\alpha_i - \varepsilon}$ instead of $(ii)$,
we can still find a set $L_{i+1}$ satisfying Property 1.
\end{PROP}

This proposition will be proven in Section 4.  Here we show how to
prove Theorem \ref{thm:main} using this proposition. We first show
how to construct sets $L_0, L_1$, and $\{X_{1,c}\}_c$. For $i = 0$,
as mentioned above, we have $L_0 = \{v_0\}$ and thus $\alpha_0 = 0$.
Note that $v_0$ has at least $\frac{C}{2 k} n^{\varepsilon}$
neighbors with edge colors from $\mathcal{C}_1$.  Let $L_1$ be these
neighbors.  Then we have $|L_1| = n^{\alpha_1} \ge \frac{C}{2k}
n^{\varepsilon}$, and so $\alpha_1 \ge \varepsilon$.  Hence $\left(
1 + \frac{\varepsilon}{2} \right) - \alpha_1 \le 1 -
\frac{\varepsilon}{2} < (1 + \varepsilon)^{-1} \left[ \left( 1 +
\frac{\varepsilon}{2} \right) - \alpha_0 \right]$.  Since $v_0$ has
at most one edge of each color, we have $|X_{1,c}| \le 1 < (8 \log n
)^1 n^{\alpha_1 - \varepsilon}$. Now we can iteratively apply
Proposition \ref{prop:key} to construct sets $L_i$ and $X_{i,c}$ for
$i=2, \cdots, k$ as long as $\alpha_{i-1} \le 1 -
\frac{1}{4}\varepsilon^2$.  Note that Property 1 above ensures that Condition (i) is always satisfied with every iteration.

\medskip

Suppose we successfully construct the sets $L_0, L_1, \hdots, L_k$ by repeatedly applying Proposition 3.1.  Recalling that $\alpha_0 = 0$, we
get
\[ \left( 1+ \frac{\varepsilon}{2} \right) - \alpha_k \le (1 + \varepsilon)^{-1} \left[ \left( 1 + \frac{\varepsilon}{2} \right) - \alpha_{k-1} \right] \le \hdots \le (1 + \varepsilon)^{-i} \left[ \left(1 + \frac{\varepsilon}{2} \right) - \alpha_0 \right], \]
and so
\[ \alpha_k \ge \left( 1 + \frac{\varepsilon}{2} \right) \left( 1 - \left(1 + \varepsilon \right)^{-k} \right). \]
Substituting $k = \left \lceil \frac{\ln 4 - \ln
\varepsilon}{\ln ( 1 + \varepsilon ) } \right \rceil$, we have
\[ \alpha_k \ge \left( 1 + \frac{\varepsilon}{2} \right) \left( 1 - \frac{1}{4} \varepsilon \right) \ge 1 + \frac{1}{8} \varepsilon, \]
and so $|L_k| = n^{\alpha_k} \ge n^{1 + \frac{1}{8} \varepsilon}$.  Thus $|L_k| > n$, which gives the necessary contradiction.

\medskip

Hence there must be some $i < k$ such that $1 - \frac{1}{4}
\varepsilon^2 < \alpha_i \le 1$.  The sizes of the sets $X_{i,c}$
satisfy
 $|X_{i,c}| \le (8 \log n)^i n^{\alpha_i - \varepsilon} = (8 \log n)^i n^{- \varepsilon} |L_i|$.
Note that the total number of colors is $m = | \mathcal{C} | < n^2$,
since there cannot be more colors than edges in $G$. Apply Lemma
\ref{lem:shrink} with $X = L_i$, subsets $X_{i,c}$ for all $c \in
\mathcal{C}$, $\beta = ( 8 \log n )^i n^{- \varepsilon}$ and $\gamma
= \frac{1}{2} n^{1 - \frac{1}{4} \varepsilon^2 - \alpha_i}$. This is
possible since
\[ 3 m \; \mathrm{exp} \left( - \frac{1}{8} \beta \gamma |L_i| \right) < 3 n^2 \; \mathrm{exp} \left( - \frac{1}{16} ( 8 \log n)^i n^{1 - \varepsilon - \frac{1}{4} \varepsilon^2} \right) < 1. \]
We obtain a set $Y \subset L_i$ such that $\frac{1}{2} \gamma |L_i| \le |Y| \le 2 \gamma |L_i|$ and $|Y \cap X_{i,c}| \le 4 \beta |Y|$ for all $c$.
Note that $ \frac{1}{4} n^{1 - \frac{1}{4} \varepsilon^2} \le |Y| \le n^{1 - \frac{1}{4} \varepsilon^2}$ and $|X_{i,c} \cap Y| \le 4 ( 8 \log n)^i |Y| n^{- \varepsilon}$.  Moreover, since we must have had $\alpha_{i-1} \le 1 - \frac{1}{4} \varepsilon^2$, we have $|Y| \ge \frac{1}{4} |L_j|$ for all $0 \le j < i$.
Let $L_i' = Y$, $|L_i'| = n^{\alpha_i'}$, and let $X_{i,c}' = X_{i,c} \cap Y$.
Then the above inequalities imply
$1 - \frac{1}{3} \varepsilon^2 < 1 - \frac{1}{4} \varepsilon^2 - \frac{2}{\log n} \le \alpha_i' \le 1 - \frac{1}{4} \varepsilon^2$, and $|X_{i,c}'| \le 4 (8 \log n)^i n^{\alpha_i' - \varepsilon}$.  We can now apply Proposition \ref{prop:key} to the sets $L_i'$ and $X_{i,c}'$. This gives the next level $L_{i+1}$ with
\[ \left( 1 + \frac{\varepsilon}{2} \right) - \alpha_{i+1} \le ( 1 + \varepsilon )^{-1} \left[ \left( 1 + \frac{\varepsilon}{2} \right) - \alpha_i' \right] \le (1 + \varepsilon)^{-1} \left[ \frac{\varepsilon}{2} + \frac{\varepsilon^2}{3} \right], \]
and so $\alpha_{i+1} \ge 1 + \frac{\varepsilon^2}{6 ( 1 + \varepsilon ) }$.  Again, this implies $|L_{i+1}| \ge n^{1 + \frac{\varepsilon^2}{6 (1 + \varepsilon)}} > n$, which is a contradiction.

\medskip

Thus $G$ must have a rainbow cycle of length at most $2k$, which completes the inductive step, and hence the proof of Theorem \ref{thm:main}.
\end{proof}

\section{Proof of Proposition \ref{prop:key}}

In this section, we furnish a proof of Proposition \ref{prop:key}. Our goal is to construct the level $L_{i+1}$ with associated sets $X_{i+1, c}$ satisfying the following properties:
\begin{enumerate}
    \item $ \left( 1 + \frac{\varepsilon}{2} \right) - \alpha_{i+1} \le ( 1 + \varepsilon )^{-1} \left[ \left( 1 + \frac{\varepsilon}{2} \right) - \alpha_i \right],$ and
    \item for all colors $c$, we have $|X_{i+1,c}| \le (8 \log n)^{i+1} n^{\alpha_{i+1} - \varepsilon}$.
\end{enumerate}

\begin{proof}[Proof of Proposition \ref{prop:key}]

Suppose that $1 \le i \le k-1$, and levels $L_j$ for $j \le i$ satisfy
Properties (i) and (ii) given in Proposition \ref{prop:key}.  Recall that by the inductive hypothesis, we know that Theorem \ref{thm:main} is true for any graph whose number of vertices $n'$ is less than $n$.  Thus we may assume that all the subgraphs of $G$ on $n'$ vertices contain at most $C [n']^{1 + \varepsilon}$ edges (otherwise we would already have a rainbow cycle of length at most $2k$).  Using this, we will show how to construct the level $L_{i+1}$ satisfying both properties.

Consider the edges of colors from $\mathcal{C}_{i+1}$ coming out of
$L_i$.  Each vertex in $L_i$ has at least $\frac{C}{2k}
n^{\varepsilon}$ such edges; importantly, we will use only
$\frac{C}{2k} n^{\varepsilon}$ of them, and disregard any additional
edges. The reason we expand the levels `slowly' in such a way is to
prevent some of the sets $X_{i,c}$ from expanding too fast. Indeed,
if we were to use all the edges, then some $X_{i,c}$ might expand
faster than we would wish, and this eventually might violate
Property 2.

Thus we have a total of $\frac{C}{2k} |L_i| n^{\varepsilon}$
edges.  If at least half of these edges went back to vertices in
$L_0 \cup L_1 \cup \hdots \cup L_{i-1}$, then the vertices in $L_0
\cup L_1 \cup \hdots \cup L_i$ would span at least $\frac{C}{4k}
|L_i| n^{\varepsilon}$ edges.  This gives us a graph on at most $4k
|L_i|$ vertices with at least $\frac{C}{4k} |L_i| n^{\varepsilon}$ edges.
By the inductive hypothesis, we have
\[ \frac{C}{4k} |L_i| n^{\varepsilon} \le C \left[4 k |L_i| \right]^{1 + \varepsilon}, \]
which is equivalent to
\[ \left( \frac{n}{|L_i|} \right)^{\varepsilon} = n^{(1 - \alpha_i) \varepsilon} \le (4 k)^{2 + \varepsilon}. \]
However, by the condition that $\alpha_i \le 1 - \frac{1}{4} \varepsilon^2$, this contradicts our bound on $n$.

Hence we may assume that at least $\frac{C}{4k} |L_i| n^{\varepsilon}$ edges go to vertices not in $L_0 \cup L_1 \cup \hdots \cup L_{i-1}$; call this set of new vertices $Y$.  Partition the vertices in $Y$ into $\log n$ sets $Y_j$, $0 \le j \le \log n - 1$, with $y \in Y_j$ if and only if $2^j \le d(y,L_i) < 2^{j+1}$ (here we are only considering edges of a color from $\mathcal{C}_{i+1}$).
By the pigeonhole principle, there is some $j^{\ast}$ such that $Y_{j^{\ast}}$ receives at least $\frac{C}{4 k \log n }|L_i|n^{\varepsilon}$ edges from $L_i$.  Let $L_{i+1} = Y_{j^{\ast}}$, and for convenience define $d = 2^{j^{\ast}}$.  As always, we will define $\alpha_{i+1}$ by $|L_{i+1}| = n^{\alpha_{i+1}}$.  Let $\delta_i = \alpha_{i+1} - \alpha_i$.

Every vertex $y \in L_{i+1}$ has degree between $d$ and $2d$ in $L_i$.
Double-counting the edges between $L_i$ and $L_{i+1}$, we have
\[ \frac{C}{4 k \log n}|L_i|n^{\varepsilon} \le e(L_i, L_{i+1}) \le 2 d |L_{i+1}|. \]
This gives
\begin{align} \label{eq:degree}
d \ge \frac{C}{8 k \log n} \frac{|L_i| n^{\varepsilon}}{|L_{i+1}| }  = \frac{C}{8k \log n} n^{\varepsilon - \delta_i}.
\end{align}

We will show below that the set $L_{i+1}$ is large enough to provide the expansion required for Property 1.  First, however, note that every vertex $y \in L_{i+1}$ can have many edges back to $L_i$.  In order to make this a level in our tree $T$, for each vertex we need to choose one edge to add to $T$.  The choice of edge induces a path from $y$ back to $v_0$, and hence these choices determine the sets $X_{i+1,c}$.
We will later show that we can choose the edges so as to satisfy Property 2 as well.

\subsection{Property 1}

We begin by providing a heuristic of the argument.
Given the level $L_i$ and the sets $X_{i,c}$,
we show that $L_{i+1}$ can be partitioned into sets $W_{c}$ such that
for every color $c$, the number of edges between $X_{i,c}$ and $W_{c}$ is
$\Omega(d|W_{c}|)$. Suppose that there exists an index
$c$ such that $|X_{i,c}| \le |W_{c}|$.
On one hand, the fact that we used only $\frac{C}{2K}n^\varepsilon$ edges
from each vertex in $X_{i,c}$ gives an upper bound on the size of $|W_{c}|$
in terms of $\delta_i$. On the other hand, the fact that we have a subgraph $G[X_{i,c}\cup W_{c}]$ which has
at most $2|W_c|$ vertices and contains at least $\Omega(d|W_{c}|)$ edges,
will by our inductive hypothesis give a lower bound on the size of $|W_{c}|$ in
terms of $\delta_i$.
By combining these bounds, we conclude that $\delta_i$ has to be quite large.

We will use Condition $(ii')$ instead of $(ii)$ in Proposition \ref{prop:key}.
Thus for all $c \in \mathcal{C}$, we have $|X_{i,c}| \le 4 (8 \log n)^{i} n^{\alpha_i - \varepsilon}$.
First we claim a rather weak bound $|L_{i+1}| > k |L_i|$.  Suppose this were not the case.
Then in the set $L_i \cup L_{i+1}$ of at most $(k+1)|L_i|$ vertices, we have at least $\frac{C}{4 k \log n}|L_i| n^{\varepsilon}$ edges.  By induction, we must have $\frac{C}{4 k \log n}|L_i|n^{\varepsilon} \le C [(k + 1) |L_i|]^{1 + \varepsilon}$, or, equivalently,
\[ \left( \frac{n}{|L_i|} \right)^{\varepsilon} = n^{(1 - \alpha_i) \varepsilon}\le 4k ( k + 1)^{1 + \varepsilon} \log n, \]
which contradicts our choice of $n$ (recall that $\alpha_i \le 1 - \frac{1}{4} \varepsilon^2$).  Thus we must have $|L_{i+1}| > k |L_i|$.

\medskip

Consider a fixed vertex $y \in L_{i+1}$, and recall that $d(y, L_i) \ge d$.  Consider any neighbor $x \in L_i$ of $y$.  The path from $v_0$ to $x$ in $T$ uses $i$ different colors $\{c_j : 1 \le j \le i \}$.
If any other neighbor $x' \in L_i$ of $y$ has a path to $v_0$
that avoids the colors $\{ c_j \}$, then we have a rainbow cycle of length $2(i+1) \le 2k$.  Thus for every neighbor $x' \in L_i$ of $y$, we must have $x' \in \cup_{j=1}^i X_{i,c_j}$.  By the pigeonhole principle, there is some $j$ such that $d(y, X_{i,c_j}) \ge \frac{d}{i}$.
Informally, this observation asserts that
every vertex $y \in L_{i+1}$ sends a large proportion of its edges to some set $X_{i,c_j}$.

For each color $c$, let $W_c$ be the set of vertices $y \in L_{i+1}$
such that $d(y,X_{i,c}) \ge \frac{d}{i}$, and note that $\{ W_c \}$
forms a cover of $L_{i+1}$.  Thus $\sum_c |W_c| \ge |L_{i+1}| > k
|L_i|$.  On the other hand, the sets $\{ X_{i,c} \}_c$ form an
$i$-fold cover of $L_i$, and so $\sum_c |X_{i,c}| = i |L_i| < k
|L_i|$.  Consequently, $\sum_c (|W_c| - |X_{i,c}|) > 0$, and so for
some particular color $c$ we have $|W_{c}| > |X_{i,c}|$. As stated
above, we will exploit the fact that there are at least $\frac{d}{i}
|W_{c}|$ edges between $W_{c}$ and $X_{i, c}$ in two different ways
to get two inequalities. Together, these will give the claimed
inequality between $\alpha_i$ and $\alpha_{i+1}$.

\medskip

First, recall that we used at most $\frac{C}{2k} n^{\varepsilon}$
edges incident to each vertex in $L_i$ to construct the set
$L_{i+1}$.  By double-counting the edges between $W_{c}$ and
$X_{i,c}$, we have
\begin{align*}
 \frac{d}{k} |W_{c}| < \frac{d}{i} |W_{c}| \le e( W_{c}, X_{i, c} ) \le \frac{C}{2k} |X_{i,c}| n^{\varepsilon},
\end{align*}
which by \eqref{eq:degree}, gives $|W_{c}| < \frac{C}{2d} |X_{i, c}|
n^{\varepsilon} \le 4 k \log n |X_{i, c}| n^{\delta_i}$. Using
Condition $(ii')$ of Proposition \ref{prop:key}, which says that
$|X_{i, c}| \le 4 (8 \log n)^i n^{\alpha_i - \varepsilon}$, we have
\begin{align} \label{eq:doublecount1}
|W_{c}| < 4 k \log n |X_{i, c}| n^{\delta_i} \le 2 k (8 \log n)^{i+1} n^{\alpha_{i+1} - \varepsilon} \le 2 k (8 \log n)^k n^{ \alpha_{i+1} - \varepsilon}.
\end{align}

Second, since there is no rainbow cycle of length at most $2k$
between $X_{i,c}$ and $W_{c}$, by the inductive hypothesis we have
\[ \frac{d}{k} |W_{c}| < e(W_{c}, X_{i,c}) < C \left[ |W_{c}| + |X_{i, c}| \right]^{ 1 + \varepsilon} < C \left[ 2 |W_{c} | \right]^{ 1 + \varepsilon}, \]
which gives $d < 2^{1+ \varepsilon} Ck |W_{c}|^{\varepsilon}$.
Hence we have
\begin{align} \label{eq:doublecount2}
\frac{C}{8 k \log n} n^{\varepsilon - \delta_i} \le d < 2^{1 + \varepsilon} Ck |W_{c}|^{\varepsilon}.
\end{align}
Combining the inequalities \eqref{eq:doublecount1} and \eqref{eq:doublecount2}, we get
\begin{eqnarray*}
 n^{\varepsilon - \delta_i} &<& 2^{4+\varepsilon} k^2 \log n |W_{c}|^\varepsilon < 2^{4 + \varepsilon} k^2 \log n \left( 2 k ( 8 \log n)^k n^{\alpha_{i+1} - \varepsilon} 
\right)^{\varepsilon} \\
&=& 2^{4 + (3k + 2) \varepsilon} k^{2 + \varepsilon} ( \log n)^{1 + k \varepsilon} n^{(\alpha_{i+1} - \varepsilon) \varepsilon}. 
\end{eqnarray*}
For our choice of $n$, we have $2^{4 + (3k + 2) \varepsilon} k^{2 + \varepsilon} ( \log n )^{1 + k \varepsilon} < n^{ \frac{1}{2} \varepsilon^2}$, and so $n^{\varepsilon - \delta_i} \le n^{\frac{1}{2} \varepsilon^2 + ( \alpha_{i+1} - \varepsilon) \varepsilon}$.  This gives $\varepsilon - \delta_i \le \frac{1}{2} \varepsilon^2 + (\alpha_{i+1} - \varepsilon) \varepsilon = \alpha_{i+1} \varepsilon - \frac{1}{2} \varepsilon^2$, which, using $\delta_i = \alpha_{i+1} - \alpha_i$, becomes
\[ \varepsilon - \alpha_{i+1} + \alpha_i \le \alpha_{i+1} \varepsilon - \frac{1}{2} \varepsilon^2. \]
Rearranging and adding $\left( 1 + \frac{\varepsilon}{2} \right)$ to both sides, we get
\[ ( 1 + \varepsilon) \left[ \left( 1 + \frac{\varepsilon}{2} \right) - \alpha_{i+1} \right] \le \left( 1 + \frac{\varepsilon}{2} \right) - \alpha_i, \]
which establishes Property 1.

\subsection{Property 2}

To obtain Property 2, we assume Condition $(ii)$ of Proposition \ref{prop:key} instead of $(ii')$.
We have shown that the next level $L_{i+1}$ is large enough.  For each of its vertices, we now need to select an edge back to $L_i$ in such a way that the sets $X_{i+1, c}$ formed satisfy the bound in Property 2.
For each $y \in L_{i+1}$, let $d_y = d(y, L_i)$.
Recall that there is a parameter $d$ such that
$d \ge \frac{C}{8k \log n} n^{\varepsilon - \delta_i}$ and $d \le d_y < 2d$ for all $y \in L_{i+1}$.
Also recall that each edge back to $L_i$ extends to a rainbow path to the root $v_0$ in the tree $T$.  For each vertex, we choose one edge uniformly at random, and show that with positive probability the resulting sets $X_{i+1,c}$ are small enough.

We can represent $|X_{i+1, c}|$ as a sum of indicator variables:
\[ |X_{i+1,c}| = \sum_{y \in L_{i+1}} \mathbf{1}_{ \{y \in X_{i+1,c}\} } . \]
Since each vertex $y$ chooses its path independently of the others, the indicator random variables in
the summand are independent.  We would first
like to obtain an estimate on $\mu_c = \BBE \left[ |X_{i+1,c}| \right]$.

First consider those $c \in \mathcal{C}_{i+1}$.  $|X_{i+1,c}|$ counts the number of times the color $c$ is used between the levels $L_i$ and $L_{i+1}$.  Since the coloring is proper, there are at most $|L_i|$ such edges.  Since all the vertices of $L_{i+1}$ have degree at least $d$, each such edge is chosen with probability at most $\frac{1}{d}$.  Thus $\mu_c \le \frac{|L_i|}{d}$, and by our bound \eqref{eq:degree} on $d$,
\[ \mu_c \le \frac{|L_i|}{d} \le \frac{8 k (\log n ) n^{\alpha_i}}{C n^{ \varepsilon - \delta_i}} < (\log n) n^{\alpha_{i+1} - \varepsilon}. \]

Now we consider those $c \notin \mathcal{C}_{i+1}$.  Note that for $y \in L_{i+1}$, we have $y \in X_{i+1,c}$ only if we choose for $y$ an edge back to $X_{i,c}$.  Thus,
\[ \mu_c = \sum_y \frac{d(y,X_{i,c})}{d_y} \le \frac{1}{d} \sum_y d(y,X_{i,c}) = \frac{1}{d} e(L_{i+1}, X_{i,c}). \]
Since all the vertices in $L_i$ send at most $\frac{C}{2k} n^{\varepsilon}$ edges into $L_{i+1}$, the above is at most
\[ \mu_c \le \frac{C}{2 k d} |X_{i,c}| n^{\varepsilon} \le \frac{C (8 \log n)^i}{2 k d} n^{\alpha_i}. \]
Using \eqref{eq:degree}, this gives $\mu_c \le \frac{1}{2} ( 8 \log n)^{i+1} n^{\alpha_{i+1} - \varepsilon}$.  Thus for $t = (8 \log n)^{i+1} n^{\alpha_{i+1} - \varepsilon}$, we have $t \ge 2 \mu_c$ for all colors $c$.

\medskip

By Theorem \ref{thm:conc}, for every color $c$, we have
\[ \BFP \left( |X_{i+1,c}| \ge t \right) \le \mathrm{exp} \left( - \frac{t}{8} \right). \]
Recalling that $\alpha_{i+1} \ge \alpha_1 \ge \varepsilon$, and $i + 1 \ge 2$, we have $t = (8 \log n)^{i+1} n^{\alpha_{i+1} - \varepsilon} \ge 64 \log n \ge 32 \ln n$.   Hence $\BFP \left( |X_{i+1,c}| \ge t \right) \le \mathrm{exp} \left( - 4 \ln n \right) = n^{-4}.$  There are at most $n^2$ colors $c$, and so a union bound gives
\[ \BFP \left( \exists c : |X_{i+1,c}| \ge t \right) \le n^2 \cdot n^{-4} = n^{-2} < 1. \]
Thus there is a choice of edges such that Property 2 holds.

\medskip

This completes the proof of Proposition \ref{prop:key}.
\end{proof}

\section{Concluding Remarks}

In this final section, we make a few remarks about our proof, and
present a couple of open problems.

\medskip

First, we note that at the beginning of our argument, we used the
Lemma \ref{lem:split} to separate the colors into disjoint classes
to be used between levels of the tree $T$.  This simplifies the
proof, at the cost of a worse constant $C(\varepsilon)$.  It is
possible to remove this step from the proof, and use most of the
edges out of a vertex at each stage.  While we would not gain much
in our argument above, this might be important if dealing with
cycles of length growing with $n$.

\medskip

Second, we noted earlier that we can strengthen our argument to
obtain rainbow cycles of length exactly $2 k$, as opposed to at most
$2 k$.  The only change that needs to be made is when establishing
Property 1 in the proof of Proposition \ref{prop:key}.  When trying
to show that every vertex in $y \in L_{i+1}$ sends a large
proportion of its edges to some set $X_{i,c}$, we first construct a
rainbow path $P_0$ of length $2( k - i - 1)$ from $y$ to some other
$y' \in L_{i+1}$, using the edges between $L_i$ and $L_{i+1}$.  Then
fix any path $P'$ from $y'$ to $v_0$ that is disjoint from $P_0$.
Note that if $y$ had a path to $v_0$ that was disjoint from $P \cup
P'$ and used a disjoint set of colors, we would have a rainbow cycle
of length $2k$.  Thus most paths from $y$ to $v_0$ must all use some
color from $P'$, which gives the desired result as before.  This
argument requires that $d$ is large relative to $k$, but if this
were not true then we would already have the desired expansion.

\medskip

Recall that $f(n)$ denotes the maximum number of edges in a rainbow
acyclic graph on $n$ vertices.  In this paper, we showed that for
any fixed $\varepsilon > 0$ and large enough $n$, $f(n) < n^{1 +
\varepsilon}$. In fact, one can use our method to obtain an upper
bound of the form $f(n) < n \; \mathrm{exp} \left( ( \log
n)^{\frac{1}{2} + \eta} \right)$ for any $\eta > 0$.  On the other
hand, the hypercube construction of Keevash, Mubayi, Sudakov and
Verstra\"{e}te gives a lower bound of $f(n) = \Omega( n \log n)$. It
would be very interesting to determine the true asymptotics of
$f(n)$. The problem of determining the rainbow Tur\'{a}n number for
even cycles also remains. It would be interesting to further narrow
the gap $\Omega \left( n^{ 1 + \frac{1}{k} } \right) \le
\mathrm{ex}^{\ast} ( n, C_{2k} ) \le O \left( n^{ 1 + \frac{(1 +
\varepsilon_k) \ln k}{k} } \right)$, and establish the order of
magnitude of the function. We believe the lower bound to be correct.

\subsection*{Acknowledgement}

We would like to thank Jan Volec for pointing out an error in our
earlier version of Proposition 3.1.

\end{document}